\documentclass[letterpaper]{amsart}

\usepackage[letterpaper]{geometry}
\usepackage{amsmath, amsthm, amsfonts, amsbsy, thmtools, amssymb}

\usepackage{thm-restate}
\usepackage{hyperref} 
\usepackage[]{ytableau}
\usepackage[mathscr]{euscript}
\usepackage{stmaryrd}
\usepackage[all, arc, curve, frame]{xy} 
\usepackage[T1]{fontenc}
\usepackage{tikz}
\usetikzlibrary{arrows}

\SetSymbolFont{stmry}{bold}{U}{stmry}{m}{n}

\title{Convergence of the Stochastic Six-Vertex Model to the ASEP} 
\author{Amol Aggarwal} 

\begin{document}

\newtheorem{thm}{Theorem}
\newtheorem{prop}[thm]{Proposition}
\newtheorem{lem}[thm]{Lemma}
\newtheorem{cor}[thm]{Corollary}
\newtheorem{conj}[thm]{Conjecture}
\newtheorem{que}[thm]{Question}
\theoremstyle{remark}
\newtheorem{rem}[thm]{Remark}
\theoremstyle{definition}
\newtheorem{definition}[thm]{Definition}
\newtheorem{exa}[thm]{Example}

\begin{abstract}

In this note we establish the convergence of the stochastic six-vertex model to the one-dimensional asymmetric simple exclusion process, under a certain limit regime recently predicted by Borodin-Corwin-Gorin. This convergence holds for arbitrary initial data. 

\end{abstract}

\maketitle

\section{Introduction}

\label{Introduction}

Over the past three decades, significant effort \cite{OCVDA, ENDDS, CFA, AASIC, ASBIC} has been devoted towards the understanding of \emph{current fluctuations} of the \emph{asymmetric simple exclusion process} (ASEP); see Section \ref{Exclusion} for definitions. When the ASEP is run under certain types of initial data for some large time $T$, these fluctuations have been proven (or are generally believed) to be of order $T^{1 / 3}$, a characteristic property of models in the \emph{Kardar-Parisi-Zhang (KPZ) universality class}. 

Central to many recent developments in this direction has been the \emph{integrability} of the ASEP, which was established by Tracy and Widom \cite{FDRA, IASEP} through a coordinate Bethe ansatz. This led to exact identities for the ASEP that were amenable to large-time asymptotics and entirely characterized the distribution of the current.

A troublesome limitation to Tracy-Widom's method (or in general to the coordinate Bethe ansatz) is that it requires some degree of guesswork to match the required identity with the initial data. This is reflected by the fact that this method has only produced asymptotic current fluctuation results for two classes of initial data (step \cite{AASIC} and step Bernoulli \cite{ASBIC}). 

There are a number of other models that are also solvable through the coordinate Bethe ansatz, but whose integrability is much better understood through alternative methods. One example is the \emph{six-vertex} (or \emph{ice}) model studied by Lieb \cite{RESI} and Baxter \cite{ESMSM}. However, in this note we consider a stochastic version of this model, due to Gwa and Spohn \cite{SVMRSASH}, called the \emph{stochastic six-vertex model}; see Section \ref{StochasticVertex} for definitions. 

Recently, Borodin, Corwin, and Gorin \cite{SSVM} implemented a coordinate Bethe ansatz (similar to Tracy-Widom's) to establish KPZ statistics for the current fluctuations of the stochastic six-vertex model. However, more recently, Borodin and Petrov \cite{HSVMRSF, IPSVMSF} found a new source of integrability for this model based on the \emph{Yang-Baxter equations}, which are a family of commutation relations between certain row-to-row and column-to-column \emph{transfer matrices} of the six-vertex model; this led to a new proof of the identities presented in \cite{SSVM}. 

Borodin and Petrov observed \cite{HSVMRSF} that these row-to-row and row and column-to-column commutation relations also allowed for a considerable degree of flexibility to the stochastic six-vertex model. In particular, one can introduce what are known as \emph{inhomogeneity parameters} and \emph{deformed spin parameters} \cite{HSVMRSF} to columns of this model that introduce different dynamics or new boundary data to the model while retaining its integrability. 

One might hope that similar parameters can be introduced for the ASEP; unfortunately, it is not known how to do this. 

However, Borodin, Corwin, and Gorin predicted in \cite{SSVM} that one can map the stochastic six-vertex model to the ASEP under a certain ``diagonal offset''\footnote{This diagonal offset is what seems to preclude the introduction of inhomogeneity and deformed spin parameters in the ASEP.} and limit degeneration; see Section \ref{ConvergenceProcesses} for a more precise statement. In particular, if true, their prediction would allow one to still use the additional parameters underlying the stochastic six-vertex model to analyze the ASEP, even though these parameters do not appear to exist directly for the ASEP.

The purpose of this note is to verify their prediction. In two separate papers, we will combine the results of this note with an analysis of a (deformed and inhomogeneous) stochastic six-vertex model to establish KPZ statistics for the ASEP under \emph{generalized step Bernoulli} \cite{PTAEPSSVMHSID} and \emph{stationary} \cite{CFSAEPSSVMCL} initial data.

This note is organized as follows. In Section \ref{ProcessModel}, we define the models and state our results (see Theorem \ref{vertexmodelprocess} and Corollary \ref{currentmodelprocess}). In Section \ref{Spin12}, we give an alternative description for the dynamics of the (offset) stochastic six-vertex model that provides a heuristic for our results. 

The remainder of this paper will be devoted to the proofs of these results; in particular, the proof will consist of three steps. 

First, we recall a definition of the ASEP in terms of \emph{time graphs} (due to Harris \cite{ASMPGM} almost 40 years ago); in Section \ref{TimeGraph}, we define these time graphs for the ASEP and also introduce certain deformations of them that give rise to a resampling of the stochastic six-vertex model. 

Second, we ``restrict'' the ASEP and stochastic six-vertex model to intervals of finite length. We explain this restriction more precisely in Section \ref{BoundedModelProcess} and justify why we can perform it	 in Section \ref{BoundedConverge}. 

Third, in Section \ref{ModelProcessConverge}, we establish the convergence of the stochastic six-vertex model to the ASEP on finite intervals by showing that the time graph (from the first step) of the stochastic six-vertex model converges to the time graph of the ASEP on compact domains.

\subsection*{Acknowledgements}

The author heartily thanks Alexei Borodin for valuable conversations and very helpful comments on a preliminary version of this note. This work was funded by the Eric Cooper and Naomi Siegel Graduate Student Fellowship I and the NSF Graduate Research Fellowship under grant number DGE1144152.

\section{Results} 

\label{ProcessModel}

In this section we state our results, given by Theorem \ref{vertexmodelprocess} and Corollary \ref{currentmodelprocess} in Section \ref{ConvergenceProcesses}. However, we first require the definition of the ASEP (given in Section \ref{Exclusion}); the stochastic six-vertex model (given in Section \ref{StochasticVertex}); its initial data (given in Section \ref{Initial}); and the associated particle system (given in Section \ref{InteractingParticles}).

\subsection{The Asymmetric Simple Exclusion Process}

\label{Exclusion}

Introduced to the mathematics community by Spitzer \cite{IMP} in 1970 (and also appearing two years earlier in the biology work of Macdonald, Gibbs, and Pipkin \cite{KBNAT}), the \emph{asymmetric simple exclusion process} (ASEP) is a continuous time Markov process that can be described as follows. 

Particles are initially (at time $0$) placed on $\mathbb{Z}$ in such a way that at most one particle occupies any site. Associated with each particle are two exponential clocks, one of rate $L$ and one of rate $R$; we assume that $R > L \ge 0$ and that all clocks are mutually independent. When some particle's left clock rings, the particle attempts to jump one space to the left; similarly, when its right clock rings, it attempts to jump one space to the right. If the destination of the jump is unoccupied, the jump is performed; otherwise it is not. This is sometimes referred to as the \emph{exclusion restriction}. 

Associated with the ASEP is an observable called the \emph{current}. To define this quantity, we \emph{tag} the particles of the ASEP, meaning that we track their evolution over time by indexing them based on initial position. Specifically, let the initial positions of the particles be $\cdots < X_{-1} (0) < X_0 (0) < X_1 (0) < \cdots$, where $X_{-1} (0) \le 0 < X_0 (0)$. The particle initially at site $X_k (0)$ will be referred to as \emph{particle $k$}. For each $k \in \mathbb{Z}$ and $t > 0$, let $X_k (t)$ denote the position of particle $k$ at time $t$. Since all jumps are nearest-neighbor, we have that $\cdots < X_{-1} (t) < X_0 (t) < X_1 (t) < \cdots $ for all $t \ge 0$.

Now, consider the ASEP after run for some time $t$. For any $x \in \mathbb{R}$, define the \emph{current through $[0, x]$} to be the almost surely finite sum 
\begin{flalign*}
J_t (x) = \displaystyle\sum_{i = - \infty}^{\infty} \big( \textbf{1}_{X_i (0) \le 0} \textbf{1}_{X_i (t) > x} - \textbf{1}_{X_i (0) > 0} \textbf{1}_{X_i (t) \le x} \big). 
\end{flalign*}

Observe that $J_t (x)$ has the following combinatorial interpretation. Color all particles initially to the right of $0$ red, and color all particles initially at or to the left of $0$ blue. Then, $J_t (x)$ denotes the number of red particles at or to the left of $x$ at time $T$ subtracted from the number of blue particles to the right of $x$ at time $t$.

\begin{figure}[t]

\begin{center}

\begin{tikzpicture}[
      >=stealth,
			]

			\draw[->, black	] (0, 0) -- (0, 4.5);
			\draw[->, black] (0, 0) -- (4.5, 0);
			\draw[->,black, thick] (0, .5) -- (.45, .5);
			\draw[->,black, thick] (0, 1) -- (.45, 1);
			\draw[->,black, thick] (0, 1.5) -- (.45, 1.5);
			\draw[->,black, thick] (0, 2) -- (.45, 2);
			\draw[->,black, thick] (0, 2.5) -- (.45, 2.5);
			\draw[->,black, thick] (0, 3) -- (.45, 3);
			\draw[->,black, thick] (0, 3.5) -- (.45, 3.5);

			\draw[->,black, thick] (.55, .5) -- (.95, .5);
			\draw[->,black, thick] (.55, 1) -- (.95, 1);
			\draw[->,black, thick] (.55, 1.5) -- (.95, 1.5);
			\draw[->,black, thick] (.55, 2.5) -- (.95, 2.5);
			\draw[->,black, thick] (.55, 3) -- (.95, 3);
			\draw[->,black, thick] (.55, 3.5) -- (.95, 3.5);
			\draw[->,black, thick] (.5, 2.05) -- (.5, 2.45);
			\draw[->,black, thick] (.5, 2.55) -- (.5, 2.95);
			\draw[->,black, thick] (.5, 3.05) -- (.5, 3.45);
			\draw[->,black, thick] (.5, 3.55) -- (.5, 3.95);
			
			\draw[->,black, thick] (1.05, .5) -- (1.45, .5);
			\draw[->,black, thick] (1.05, 1) -- (1.45, 1);
			\draw[->,black, thick] (1.05, 2) -- (1.45, 2);
			\draw[->,black, thick] (1.05, 2.5) -- (1.45, 2.5);
			\draw[->,black, thick] (1.05, 3.5) -- (1.45, 3.5);
			\draw[->,black, thick] (1, 1.55) -- (1, 1.95);
			\draw[->,black, thick] (1, 3.05) -- (1, 3.45);
			\draw[->,black, thick] (1, 3.55) -- (1, 3.95);
			
			\draw[->,black, thick] (1.5, .55) -- (1.5, .95);
			\draw[->,black, thick] (1.5, 1.05) -- (1.5, 1.45);
			\draw[->,black, thick] (1.5, 1.55) -- (1.5, 1.95);
			\draw[->,black, thick] (1.5, 2.05) -- (1.5, 2.45);
			\draw[->,black, thick] (1.5, 2.55) -- (1.5, 2.95);
			\draw[->,black, thick] (1.5, 3.55) -- (1.5, 3.95);
			\draw[->,black, thick] (1.55, 1) -- (1.95, 1);
			\draw[->,black, thick] (1.55, 2) -- (1.95, 2);
			\draw[->,black, thick] (1.55, 2.5) -- (1.95, 2.5);
			\draw[->,black, thick] (1.55, 3) -- (1.95, 3);

			\draw[->,black, thick] (2, 1.05) -- (2, 1.45);
			\draw[->,black, thick] (2, 3.05) -- (2, 3.45);
			\draw[->,black, thick] (2, 3.55) -- (2, 3.95);
			\draw[->,black, thick] (2.05, 1.5) -- (2.45, 1.5);
			\draw[->,black, thick] (2.05, 2) -- (2.45, 2);
			\draw[->,black, thick] (2.05, 2.5) -- (2.45, 2.5);

			\draw[->,black, thick] (2.5, 1.55) -- (2.5, 1.95);
			\draw[->,black, thick] (2.5, 2.05) -- (2.5, 2.45);
			\draw[->,black, thick] (2.5, 2.55) -- (2.5, 2.95);
			\draw[->,black, thick] (2.5, 3.05) -- (2.5, 3.45);
			\draw[->,black, thick] (2.5, 3.55) -- (2.5, 3.95);
			\draw[->,black, thick] (2.55, 2) -- (2.95, 2);
			\draw[->,black, thick] (2.55, 2.5) -- (2.95, 2.5);

			\draw[->,black, thick] (3, 2.55) -- (3, 2.95);
			\draw[->,black, thick] (3.05, 2) -- (3.45, 2);
			\draw[->,black, thick] (3.05, 3) -- (3.45, 3);

			\draw[->,black, thick] (3.5, 2.05) -- (3.5, 2.45);
			\draw[->,black, thick] (3.5, 2.55) -- (3.5, 2.95);
			\draw[->,black, thick] (3.5, 3.05) -- (3.5, 3.45);
			\draw[->,black, thick] (3.5, 3.55) -- (3.5, 3.95);
			\draw[->,black, thick] (3.55, 3) -- (3.95, 3);
		
			\draw[->,black, thick] (4, 3.05) -- (4, 3.45);
			\draw[->,black, thick] (4, 3.55) -- (4, 3.95);

			\filldraw[fill=gray!50!white, draw=black] (.5, .5) circle [radius=.05];
			\filldraw[fill=gray!50!white, draw=black] (.5, 1) circle [radius=.05];
			\filldraw[fill=gray!50!white, draw=black] (.5, 1.5) circle [radius=.05];
			\filldraw[fill=gray!50!white, draw=black] (.5, 2) circle [radius=.05];
			\filldraw[fill=gray!50!white, draw=black] (.5, 2.5) circle [radius=.05];
			\filldraw[fill=gray!50!white, draw=black] (.5, 3) circle [radius=.05];
			\filldraw[fill=gray!50!white, draw=black] (.5, 3.5) circle [radius=.05];

			\filldraw[fill=gray!50!white, draw=black] (1, .5) circle [radius=.05];
			\filldraw[fill=gray!50!white, draw=black] (1, 1) circle [radius=.05];
			\filldraw[fill=gray!50!white, draw=black] (1, 1.5) circle [radius=.05];
			\filldraw[fill=gray!50!white, draw=black] (1, 2) circle [radius=.05];
			\filldraw[fill=gray!50!white, draw=black] (1, 2.5) circle [radius=.05];
			\filldraw[fill=gray!50!white, draw=black] (1, 3) circle [radius=.05];
			\filldraw[fill=gray!50!white, draw=black] (1, 3.5) circle [radius=.05];
			
			\filldraw[fill=gray!50!white, draw=black] (1.5, .5) circle [radius=.05];
			\filldraw[fill=gray!50!white, draw=black] (1.5, 1) circle [radius=.05];
			\filldraw[fill=gray!50!white, draw=black] (1.5, 1.5) circle [radius=.05];
			\filldraw[fill=gray!50!white, draw=black] (1.5, 2) circle [radius=.05];
			\filldraw[fill=gray!50!white, draw=black] (1.5, 2.5) circle [radius=.05];
			\filldraw[fill=gray!50!white, draw=black] (1.5, 3) circle [radius=.05];
			\filldraw[fill=gray!50!white, draw=black] (1.5, 3.5) circle [radius=.05];
			
			\filldraw[fill=gray!50!white, draw=black] (2, .5) circle [radius=.05];
			\filldraw[fill=gray!50!white, draw=black] (2, 1) circle [radius=.05];
			\filldraw[fill=gray!50!white, draw=black] (2, 1.5) circle [radius=.05];
			\filldraw[fill=gray!50!white, draw=black] (2, 2) circle [radius=.05];
			\filldraw[fill=gray!50!white, draw=black] (2, 2.5) circle [radius=.05];
			\filldraw[fill=gray!50!white, draw=black] (2, 3) circle [radius=.05];
			\filldraw[fill=gray!50!white, draw=black] (2, 3.5) circle [radius=.05];
			
			\filldraw[fill=gray!50!white, draw=black] (2.5, .5) circle [radius=.05];
			\filldraw[fill=gray!50!white, draw=black] (2.5, 1) circle [radius=.05];
			\filldraw[fill=gray!50!white, draw=black] (2.5, 1.5) circle [radius=.05];
			\filldraw[fill=gray!50!white, draw=black] (2.5, 2) circle [radius=.05];
			\filldraw[fill=gray!50!white, draw=black] (2.5, 2.5) circle [radius=.05];
			\filldraw[fill=gray!50!white, draw=black] (2.5, 3) circle [radius=.05];
			\filldraw[fill=gray!50!white, draw=black] (2.5, 3.5) circle [radius=.05];
			
			\filldraw[fill=gray!50!white, draw=black] (3, .5) circle [radius=.05];
			\filldraw[fill=gray!50!white, draw=black] (3, 1) circle [radius=.05];
			\filldraw[fill=gray!50!white, draw=black] (3, 1.5) circle [radius=.05];
			\filldraw[fill=gray!50!white, draw=black] (3, 2) circle [radius=.05];
			\filldraw[fill=gray!50!white, draw=black] (3, 2.5) circle [radius=.05];
			\filldraw[fill=gray!50!white, draw=black] (3, 3) circle [radius=.05];
			\filldraw[fill=gray!50!white, draw=black] (3, 3.5) circle [radius=.05];
			
			\filldraw[fill=gray!50!white, draw=black] (3.5, .5) circle [radius=.05];
			\filldraw[fill=gray!50!white, draw=black] (3.5, 1) circle [radius=.05];
			\filldraw[fill=gray!50!white, draw=black] (3.5, 1.5) circle [radius=.05];
			\filldraw[fill=gray!50!white, draw=black] (3.5, 2) circle [radius=.05];
			\filldraw[fill=gray!50!white, draw=black] (3.5, 2.5) circle [radius=.05];
			\filldraw[fill=gray!50!white, draw=black] (3.5, 3) circle [radius=.05];
			\filldraw[fill=gray!50!white, draw=black] (3.5, 3.5) circle [radius=.05];
			
			\filldraw[fill=gray!50!white, draw=black] (4, .5) circle [radius=.05];
			\filldraw[fill=gray!50!white, draw=black] (4, 1) circle [radius=.05];
			\filldraw[fill=gray!50!white, draw=black] (4, 1.5) circle [radius=.05];
			\filldraw[fill=gray!50!white, draw=black] (4, 2) circle [radius=.05];
			\filldraw[fill=gray!50!white, draw=black] (4, 2.5) circle [radius=.05];
			\filldraw[fill=gray!50!white, draw=black] (4, 3) circle [radius=.05];
			\filldraw[fill=gray!50!white, draw=black] (4, 3.5) circle [radius=.05];

\end{tikzpicture}

\end{center}	

\caption{\label{figurevertexwedge} A sample of the stochastic six-vertex model with step boundary data is depicted above.} 
\end{figure}
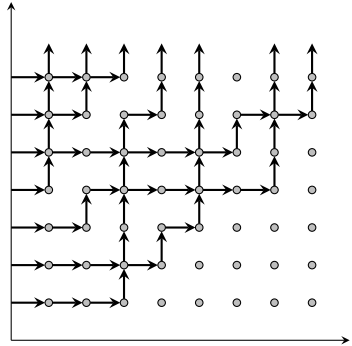

\subsection{The Stochastic Six-Vertex Model}

\label{StochasticVertex}

A \emph{six-vertex directed path ensemble} is a family up-right directed paths in the non-negative quadrant $\mathbb{Z}_{> 0}^2$, such that each path emanates from either the $x$-axis or $y$-axis, and such that no two paths share an edge (although they may share vertices); see Figure \ref{figurevertexwedge}. In particular, each vertex has six possible \emph{arrow configurations}, which are listed in the top row of Figure \ref{sixvertexfigure}. \emph{Initial data}, or \emph{boundary data}, for such an ensemble is prescribed by dictating which vertices on the positive $x$-axis and positive $y$-axis are entrance sites for a directed path. One example of initial data is \emph{step initial data}, in which paths only enter through the $y$-axis, and all vertices on the $y$-axis are entrance sites for paths; see Figure \ref{figurevertexwedge}. 

Now, fix parameters $\delta_1, \delta_2 \in [0, 1)$ and some initial data. The \emph{stochastic six-vertex model} $\mathcal{P} = \mathcal{P} (\delta_1, \delta_2)$ will be the infinite-volume limit of a family of probability measures $\mathcal{P}_n = \mathcal{P}_n (\delta_1, \delta_2)$ defined on the set of six-vertex directed path ensembles whose vertices are all contained in triangles of the form $\mathbb{T}_n = \{ (x, y) \in \mathbb{Z}_{\ge 0}^2: x + y \le n \}$. The first such probability measure $\mathcal{P}_1$ is supported by the empty ensemble.

For each positive integer $n$, we define $\mathcal{P}_{n + 1}$ from $\mathcal{P}_n$ through the following Markovian update rules. Use $\mathcal{P}_n$ to sample a directed path ensemble $\mathcal{E}_n$ on $\mathbb{T}_n$. This (and the initial data) gives arrow configurations (of the type shown in Figure \ref{sixvertexfigure}) to all vertices in the positive quadrant strictly below the diagonal $\mathbb{D}_n = \{ (x, y) \in \mathbb{Z}_{> 0}^2: x + y = n \}$. Each vertex on $\mathbb{D}_n$ is also given ``half'' of an arrow configuration, in the sense that it is given the directions of all entering paths but no direction of any exiting path. 

To extend $\mathcal{E}_n$ to a path ensemble on $\mathbb{T}_{n + 1}$, we must ``complete'' the configurations (specify the exiting paths) at all vertices $(x, y) \in \mathbb{D}_n$. Any half-configuration can be completed in at most two ways; selecting between these completions is done randomly, according to the probabilities given in the second row of Figure \ref{sixvertexfigure}. 

In this way, we obtain a random ensemble $\mathcal{E}_{n + 1}$ on $\mathbb{T}_{n + 1}$; the resulting probability measure on path ensembles with vertices in $\mathbb{T}_{n + 1}$ is denoted $\mathcal{P}_{n + 1}$. Now, set $\mathcal{P} = \lim_{n \rightarrow \infty} \mathcal{P}_n$. 

As in the ASEP, there exists an observable of interest for stochastic six-vertex model called the \emph{current}. To define it, we color a path red if it emanates from the $x$-axis, and we color a path blue if it emanates from the $y$-axis. Let $(X, Y) \in \mathbb{R}_{> 0}^2$. We define the \emph{current} (or \emph{height function}) $\mathfrak{H} (X, Y)$ of the stochastic six-vertex model at $(X, Y)$ to be the number of red paths that intersect the line $y = Y$ at or to the left of $(X, Y)$ subtracted from the number of blue paths that intersect the line $y = Y$ to the right of $(X, Y)$.

\begin{figure}[t]

\begin{center}

\begin{tikzpicture}[
      >=stealth,
			scale = .7
			]

			\draw[-, black] (-7.5, -.8) -- (7.5, -.8);
			\draw[-, black] (-7.5, 0) -- (7.5, 0);
			\draw[-, black] (-7.5, 2) -- (7.5, 2);
			\draw[-, black] (-7.5, -.8) -- (-7.5, 2);
			\draw[-, black] (7.5, -.8) -- (7.5, 2);
			\draw[-, black] (-5, -.8) -- (-5, 2);
			\draw[-, black] (5, -.8) -- (5, 2);
			\draw[-, black] (-2.5, -.8) -- (-2.5, 2);
			\draw[-, black] (2.5, -.8) -- (2.5, 2);
			\draw[-, black] (0, -.8) -- (0, 2);

			\draw[->, black,  thick] (-3.75, 1.1) -- (-3.75, 1.9);
			\draw[->, black,  thick] (-3.75, .1) -- (-3.75, .9);

			\draw[->, black,  thick] (-1.25, .1) -- (-1.25, .9);
			\draw[->, black,  thick] (-1.15, 1) -- (-.35, 1);

			\draw[->, black,  thick] (1.35, 1) -- (2.15, 1);
			\draw[->, black,  thick] (.35, 1) -- (1.15, 1);
			
			\draw[->, black,  thick] (3.75, 1.1) -- (3.75, 1.9);
			\draw[->, black,  thick] (2.85, 1) -- (3.65, 1);
			
			\draw[->, black,  thick] (6.25, 1.1) -- (6.25, 1.9);
			\draw[->, black,  thick] (6.25, .1) -- (6.25, .9);
			\draw[->, black,  thick] (6.35, 1) -- (7.15, 1);
			\draw[->, black,  thick] (5.35, 1) -- (6.15, 1);
				
			\filldraw[fill=gray!50!white, draw=black] (-6.25, 1) circle [radius=.1] node [black,below=21] {$1$};
			\filldraw[fill=gray!50!white, draw=black] (-3.75, 1) circle [radius=.1] node [black,below=21] {$\delta_1$};
			\filldraw[fill=gray!50!white, draw=black] (-1.25, 1) circle [radius=.1] node [black,below=21] {$1 - \delta_1$};
			\filldraw[fill=gray!50!white, draw=black] (1.25, 1) circle [radius=.1] node [black,below=21] {$\delta_2$};
			\filldraw[fill=gray!50!white, draw=black] (3.75, 1) circle [radius=.1] node [black,below=21] {$1 - \delta_2$};
			\filldraw[fill=gray!50!white, draw=black] (6.25, 1) circle [radius=.1] node [black,below=21] {$1$};

\end{tikzpicture}

\end{center}

\caption{\label{sixvertexfigure} The top row in the chart shows the six possible arrow configurations at vertices in the stochastic six-vertex model; the bottom row shows the corresponding probabilities. }
\end{figure}

\subsection{Initial Data}

\label{Initial}

In this section we describe initial data for the ASEP and stochastic six-vertex model. First we address the stochastic six-vertex model. 

\begin{definition}

\label{initialxy} 

Let $\varphi = \big( \varphi (1), \varphi (2), \ldots \big)$, with $\varphi (t) = \big( \varphi_t^{(x)}, \varphi_t^{(y)} \big) \in \{ 0, 1 \} \times \{ 0, 1 \} $, be a stochastic process. In the stochastic six-vertex model with \emph{initial data $\varphi$}, the site $(1, t)$ is an entrance site for a path if and only if $\varphi_t^{(y)} = 1$, and the site $(t, 1)$ is an entrance site for a path if and only if $\varphi_t^{(x)} = 1$. 
\end{definition}

For example, \emph{double-sided Bernoulli initial data} with \emph{parameters} $b_1$ and $b_2$ (also called \emph{double sided $(b_1, b_2)$-Bernoulli initial data}) arises when the $\varphi (i)$ are all mutually independent, such that $\varphi_i^{(y)}$ and $\varphi_i^{(x)}$ are independent $0-1$ Bernoulli random variables, with means $b_1$ and $b_2$, respectively. 

Under this initial data, paths can enter from both the $x$-axis and the $y$-axis; each vertex on the positive $x$-axis is an entrance point for a path with probability $b_2$, and each vertex on the positive $y$-axis is an entrance point for a path with probability $b_1$. Here, the $\varphi (i)$ are mutually independent. However, they need not be in general; see, for instance, Definition 3.3 of \cite{PTAEPSSVMHSID}. 

Now let us define initial data for the ASEP. 

\begin{definition}

\label{initialprocess}

Let $\varphi = (\varphi (1), \varphi (2), \ldots )$ be a stochastic process, where $\varphi (i) = \big( \varphi_i^{(x)}, \varphi_i^{(y)} \big) \in \{ 0, 1 \} \times \{ 0, 1 \}$. Define the \emph{ASEP with initial data $\varphi$} to be the ASEP with site $i \in \mathbb{Z}_{> 0}$ initially occupied if and only if $\varphi_i^{(x)} = 1$, and with site $i \in \mathbb{Z}_{\le 0}$ initially occupied if and only if $\varphi_{1 - i}^{(y)} = 1$. 

\end{definition}

One can define double-sided Bernoulli initial data for the ASEP in the same way as done for the stochastic six-vertex model above. Now, this means that sites at or to the left of $0$ are initially occupied with probability $b_1$ and that sites to the right of $0$ are occupied with probability $b_2$.

\subsection{Vertex Models and Interacting Particle Systems}

\label{InteractingParticles}

Let $\mathcal{P}^{(T)}$ denote the restriction of the random six-vertex path ensemble (given by the measure $\mathcal{P} = \mathcal{P} (\delta_1, \delta_2)$ from Section \ref{StochasticVertex}) to the strip $\mathbb{Z}_{> 0} \times [0, T]$. We will use the probability measure $\mathcal{P}^{(T)}$ to produce a discrete-time interacting particle system on $\mathbb{Z}_{> 0}$, defined up to time $T - 1$, as follows. 

Sample a line ensemble $\mathcal{E}$ randomly under $\mathcal{P}^{(T)}$, and consider the arrow configuration it associates to some vertex $(p, t) \in \mathbb{Z}_{> 0} \times [1, T]$. We will place a particle at site $p$ and time $t - 1$ if and only if a path vertically enters through the vertex $(p, t)$ (so if and only if the arrow configuration at $(p, t)$ is given by either the second or third type in Figure \ref{sixvertexfigure}). Therefore, the paths in the path ensemble $\mathcal{E}$ correspond to space-time trajectories of the particles. 

Observe in particular that paths entering through the $x$-axis correspond to the particles that are ``initially'' in the system, and that paths entering through the $y$-axis correspond to particles that ``enter'' the system after time $0$. Particles of the former and latter types will be colored red and blue, respectively (in analogy with the coloring described at the end of Section \ref{StochasticVertex}). 

As in the ASEP, we \emph{tag} these particles as follows. At any time $t \ge 0$, let the positions of the particles be $\cdots < p_{-1} (t) < p_0 (t) < p_1 (t) < \cdots $, where $p_0 (t)$ refers to the position of the leftmost red particle and $p_{-1} (t)$ refers to the position of the rightmost blue particle. Observe that there are finitely many (denoted $N(t) \le t$) red particles in the system at time $t$. Thus, we may rewrite the locations of the particles as $p (t) = \big( p_{-N(t)} (t), p_{1 - N(t)} (t), \ldots \big)$.

\subsection{Convergence Results} 

\label{ConvergenceProcesses}

In this section we explain the result of this paper, which is a limit degeneration under which the stochastic six-vertex model converges to the ASEP. This was originally predicted by Borodin, Corwin, and Gorin \cite{SSVM} (see equation (6) of Section 2.2 of that paper). We will explain their heuristic behind this result later, in Section \ref{Spin12}. 

\begin{thm}

\label{vertexmodelprocess} 

Fix real numbers $R, L \ge 0$ and a stochastic process $\varphi = \big( \varphi (1), \varphi (2), \ldots \big)$, with $\varphi (i) = \big( \varphi_i^{(x)}, \varphi_i^{(y)} \big) \in \{ 0, 1 \} \times \{ 0, 1 \}$. Consider the ASEP with left jump rate $L$, right jump rate $R$, and initial data $\varphi$. Denote the tagged particles for this ASEP by $\cdots < X_{-1} (t) < X_0 (t) < X_1 (t) < \cdots $. 

Let $\varepsilon > 0$ be a real number, and denote $\delta_1 = \delta_{1; \varepsilon} = \varepsilon L$ and $\delta_2 = \delta_{2; \varepsilon} = \varepsilon R$; assume that $\delta_1, \delta_2 \in [0, 1)$. Consider the stochastic six-vertex model $\mathcal{P} (\delta_1, \delta_2)$, with initial data $\varphi$. Let $p(t) = \big( p_{-N(t)} (t), p_{1 - N(t)} (t), p_{2 - N(t)} (t), \ldots \big)$ denote the particle configuration of this model at time $t$.  

Fix a positive integer $k$; integers $i_1, i_2, \ldots , i_k$; positive real numbers $t_1, t_2, \ldots , t_k > 0$; and a finite subset $S \subset \mathbb{Z}$. Denoting $q_t (i) = p_t (i) - t$ for each $i, t \in \mathbb{Z}$, we have that 	
\begin{flalign}
\label{exclusionvertexconvergencefiniteset}
\displaystyle\lim_{\varepsilon \rightarrow 0} \mathbb{P} \Big[ q_{i_1} (\lfloor \varepsilon^{-1} t_1 \rfloor), q_{i_2} (\lfloor \varepsilon^{-1} t_2 \rfloor), \ldots , q_{i_k} (\lfloor \varepsilon^{-1} t_k \rfloor) \in S \big] = \mathbb{P} \Big[ X_{i_1} (t_1), X_{i_2} (t_2), \ldots , X_{i_k} (t_k) \in S \Big].
\end{flalign}

\end{thm}

Observe that the stochastic six-vertex model in Theorem \ref{vertexmodelprocess} is \emph{offset by the diagonal}, meaning that we consider the \emph{offset particle positions} $q_i (t) = p_i (t) - t$ instead of the original particle positions $p_i (t)$. Theorem \ref{vertexmodelprocess} suggests that the current of this offset model should converge to the current of the ASEP; this is stated explicitly in the following corollary. 

\begin{cor}

\label{currentmodelprocess}

Fix real numbers $R, L \ge 0$ and a stochastic process $\varphi = \big( \varphi (1), \varphi (2), \ldots \big)$, with $\varphi (i) = \big( \varphi_i^{(x)}, \varphi_i^{(y)} \big) \in \{ 0, 1 \} \times \{ 0, 1 \}$. Let $\varepsilon > 0$ be a real number, and denote $\delta_1 = \delta_{1; \varepsilon} = \varepsilon L$ and $\delta_2 = \delta_{2; \varepsilon} = \varepsilon R$; assume that $\delta_1, \delta_2 \in (0, 1)$. Also fix $r \in \mathbb{R}$, $t \in \mathbb{R}_{> 0}$, and $x \in \mathbb{Z}$.

Let $p_{\varepsilon} (x; r) = \mathbb{P}_V \big[ \mathfrak{H} \big( x + \lfloor \varepsilon^{-1} t \rfloor, \lfloor \varepsilon^{-1} t \rfloor \big) \ge r \big]$, where the probability $\mathbb{P}_V$ is under the stochastic six-vertex model $\mathcal{P} (\delta_1, \delta_2)$ with initial data $\varphi$. Furthermore, let $p (x; r) = \mathbb{P}_A \big[ J_t \big( x \big) \ge r \big]$, where the probability $\mathbb{P}_A$ is under the ASEP with left jump rate $L$, right jump rate $R$, and initial data $\varphi$. 

Then, $\lim_{\varepsilon \rightarrow 0} p_{\varepsilon} (x; r) = p (x; r)$. 

\end{cor}

\begin{proof}[Proof of Corollary \ref{currentmodelprocess} Assuming Theorem \ref{vertexmodelprocess}] 

This corollary can be deduced from Theorem \ref{vertexmodelprocess} and the fact that the types of \emph{cylinder functions} discussed in the statement of Theorem \ref{vertexmodelprocess} define the ASEP and stochastic six-vertex model completely. 

However, let us suggest an alternative proof. Observe that 
\begin{flalign}
\label{sumva}
p(x; r) = \displaystyle\sum_{k = x + 1}^{\infty} \mathbb{P}_A \big[ X_{-r} (t) = k \big]; \qquad p_{\varepsilon} (x; r) = \displaystyle\sum_{k = x + 1}^{\infty} \mathbb{P}_V \big[ p_{-r} (\lfloor \varepsilon^{-1} t \rfloor) = k + \lfloor \varepsilon^{-1} t \rfloor \big]. 
\end{flalign}

By Theorem \ref{vertexmodelprocess}, each summand in the first sum of \eqref{sumva} converges to the corresponding summand in the second sum in \eqref{sumva}. Thus, it suffices to establish a uniform, exponential (in $|k|$)	 estimate on probabilities of the form $\mathbb{P}_V \big[ p_{-r} (\lfloor \varepsilon^{-1} t \rfloor) = k + \lfloor \varepsilon^{-1} t \rfloor \big]$. 

To that end, we first assume that $k \ge 0$. Observe that the event that $p_{-r} (\lfloor \varepsilon^{-1} t \rfloor) = k + \lfloor \varepsilon^{-1} t \rfloor$ is contained in the event that there exist at least $k$ vertices, below the line $y = \lfloor \varepsilon^{-1} t \rfloor + 1$ and on the path corresponding to particle $-r$, that all share the fourth (probability $\delta_2$) arrow configuration depicted in Figure \ref{sixvertexfigure}. The probability of this occurring is at most $\binom{\lfloor \varepsilon^{-1} t \rfloor}{k} \delta_2^k \le k!^{-1} (tR)^k$, from which it follows that $\mathbb{P}_V \big[ p_{-r} (\lfloor \varepsilon^{-1} t \rfloor) = k + \lfloor \varepsilon^{-1} t \rfloor \big] \le k!^{-1} (tR)^k$ when $k \ge 0$. 

Similarly, $\mathbb{P}_V \big[ p_{-r} (\lfloor \varepsilon^{-1} t \rfloor) = k + \lfloor \varepsilon^{-1} t \rfloor \big] \le k!^{-1} (tL)^k$ when $k < 0$. Hence, in either case, $\mathbb{P}_V \big[ p_{-r} (\lfloor \varepsilon^{-1} t \rfloor) = k + \lfloor \varepsilon^{-1} t \rfloor \big]$ decays exponentially in $|k|$ (independently of $\varepsilon > 0$).  

Thus, the corollary follows from \eqref{sumva}, Theorem \ref{vertexmodelprocess}, and the dominated convergence theorem. 
\end{proof}

\section{The Offset Stochastic Six-Vertex Model}

\label{Spin12}

Recall from Section \ref{InteractingParticles} that the stochastic six-vertex model can be viewed as an interacting particle system. In this section we explicitly describe the dynamics of the \emph{offset model} $q(t) = (q_{-N(t)} (t), q_{1 - N(t)} (t), \ldots )$ (where $q_i (t) = p_i (t) - t$), which is a particle system on $\mathbb{Z}$ instead of $\mathbb{Z}_{> 0}$. We omit the verification that these are indeed the dynamics of $q$, as this amounts to recalling the definition of the stochastic six-vertex model as an interacting particle system (explained in Section \ref{InteractingParticles}; see also Section 2.2 of \cite{SSVM}) and appropriately shifting to the right. 

In what follows we refer to the particle at site $q_i (t)$, at time $t$, as \emph{particle $i$.} 

\begin{enumerate}

\item{ At time $t = 0$, place a (red) particle at site $i \in \mathbb{Z}_{\ge 1}$ if and only if $\varphi_i^{(x)} = 1$. Then order the positions of the particles $q_0 (0) < q_1 (0) < \cdots$. Also set $N(0) = 0$. }

\item{ At time $t \ge 1$, set $N(t) = N(t - 1) + 1$ if and only if $\varphi_t^{(y)} = 1$; that is, add a new (blue) particle to the system at time $t$ if and only if $\varphi_t^{(y)} = 1$. 

\begin{enumerate}
\item{ If $\varphi_t^{(y)} = 1$, then $\mathbb{P} \big[ q_{-N(t)} (t) = j + 1 - t \big| q(t - 1) \big] = (1 - \delta_2) \delta_2^j$ for each integer $j \in [0, k - 1]$, and $\mathbb{P} \big[ q_{-N(t)} = k + 1 - t \big| q(t - 1) \big] = \delta_2^k$; here, $k = t + q_{-N(t - 1)} (t - 1) - 2$. Equivalently, if a new particle is added to the system at time $t$, then it is placed at site $1 - t$ and then jumps to the right according to a $\delta_2$-geometric distribution constrained to stay to the left of particle $-N(t - 1)$. }

\end{enumerate}
}

\item{ Now let $i > - N(t)$. Conditioned on $q (t - 1)$ and $q_{i - 1} (t)$, particle $i$ jumps as follows. 

\begin{enumerate}

\item{ If $q_{i - 1} (t) = q_i (t - 1) - 1$, then $\mathbb{P} \big[ q_i (t) = q_i (t - 1) + j | q(t - 1), q_{i - 1} (t) \big] = (1 - \delta_2) \delta_2^j$ for each $j \in [0, k - 1]$, and $\mathbb{P} \big[ q_i (t) = q_i (t - 1) + k | q(t - 1), q_{i - 1} (t) \big] = \delta_2^k$; here, $k = q_{i + 1} (t - 1) - q_i (t - 1) - 1$. Equivalently, if the site directly to the left of particle $i$ is occupied, then it jumps to the right under a $\delta_2$-geometric distribution constrained to stay left of particle $i + 1$. }

\item{ If $q_{i - 1} (t) < q_i (t - 1) - 1$ (meaning that the site directly to the left of particle $i$ is unoccupied), then particle $i$ jumps as follows. 

\begin{enumerate}

\item{ We have that $\mathbb{P} \big[ q_i (t) = q_i (t - 1) - 1 \big| q(t - 1), q_{i - 1} (t) \big] = \delta_1$. Equivalently,  particle $i$ jumps to the left with probability $\delta_1$.}

\item{ We have that $\mathbb{P} \big[ q_i (t) = q_i (t - 1) + j \big| q(t - 1), q_{i - 1} (t) \big] = (1 - \delta_1) (1 - \delta_2) \delta_2^j$ for each $j \in [0, k - 1]$, and $\mathbb{P} \big[ q_i (t) = q_i (t - 1) + k \big] = (1 - \delta_1) \delta_2^k$; here, $k = q_{i + 1} (t - 1) - q_i (t - 1) - 1$. Equivalently, if particle $i$ did not jump left, then it jumps to the right under a $\delta_2$-geometric distribution constrained to stay to the left of particle $i + 1$. }

\end{enumerate}  }

\end{enumerate}
}

\end{enumerate}

Now, suppose we set $\delta_1 = \varepsilon L$, $\delta_2 = \varepsilon R$, scale time by $\varepsilon^{-1}$, and send $\varepsilon$ to $0$. Then, the dynamics of any single particle converges to the dynamics of an ASEP particle, in the same position and surrounding particle configuration. This provides some heuristic for Theorem \ref{vertexmodelprocess}, which was originally given in Section 2.2 of \cite{SSVM} (see also Section 6.5 of \cite{HSVMRSF}). In fact, this observation would result in a direct proof of this theorem if both processes had only finitely many particles. 

However, the theorem does not follow in the same way when both systems have infinitely many particles, due to the potential interaction of particles that are far away. In fact, the same issue originally posed trouble 45 years ago when probabilists were attempting to define the ASEP. In the next section we explain how this was overcome by providing a precise definition of the ASEP.

\section{Time Graphs} 

\label{TimeGraph}

In this section we introduce the notion of \emph{time graphs}, which were originally used by Harris \cite{ASMPGM} to define the ASEP with infinitely many particles. We begin in Section \ref{InfiniteExclusion} by defining time graphs for the ASEP. In Section \ref{OffsetGraph}, we introduce an analogous time graph for the stochastic six-vertex model. These time graphs will be used to exhibit couplings (in Section \ref{BoundedConverge}) that will facilitate the proof of Theorem \ref{vertexmodelprocess}.

\subsection{Defining the ASEP With Infinitely Many Particles}

\label{InfiniteExclusion}

Recall that, in the stochastic six-vertex model, the order in which particles jump is fixed; left particles jump before right particles. In contrast, the dynamics of the ASEP are not sequential; there is no specified order in which particles must jump. Although almost surely no particles jump simultaneously, it could be the case that infinitely many clocks will ring in any arbitrarily small time interval. Thus, it might not be immediately apparent that the ASEP is well-defined with infinitely many particles, due to the possibility that far apart particles can interact. 

In fact, shortly after Spitzer introduced the ASEP, a number of papers were published confirming its well-posedness; see for instance the works of Holley \cite{CIIPS}, Harris \cite{ASMPGM, NMIP}, and Liggett \cite{ETPS}. Let us explain how one of these proofs works, for it will be useful to us later. 

In particular, we will describe the \emph{graphical} definition of the ASEP, which is due to Harris; see \cite{ASMPGM} or Section 2.1.1 of Corwin's survey \cite{EUC}. Instead of associating clocks to each particle, we associate mutually independent right and left exponential clocks (of rates $R$ and $L$, respectively) to each site $i \in \mathbb{Z}$. For each $i \in \mathbb{Z}$, let $\mathcal{R}^{(i)}$ denote the set of times $t > 0$ at which site $i$'s right clock rings, and let $\mathcal{L}^{(i)}$ denote the set of times $t > 0$ at which site $i$'s left clock rings. 

Now let $\mathcal{G} = \bigsqcup_{i \in \mathbb{Z}} \big( \bigsqcup_{t \in \mathcal{R}^{(i)}} \{ (t; i, i + 1) \} \cup \bigsqcup_{t \in \mathcal{L}^{(i)}} \{ (t; i, i - 1) \} \big) \subset \mathbb{R}_{> 0} \times \mathbb{Z} \times \mathbb{Z}$. We refer to $\mathcal{G}$ as the \emph{time graph} of the ASEP. The term ``time graph'' is used since it is typically associated with a directed graph on $\mathbb{Z} \times \mathbb{R}_{> 0}$, in which there is a directed edge from $(i_1, t_1)$ to $(i_2, t_2)$ if and only if $t_1 = t_2$ and $(t_1; i_1, i_2) \in \mathcal{G}$; we will not require this pictorial representation here, but see Figure 6 in Section 2.1.1 of \cite{EUC} for an example. 

The elements of $\mathcal{G}$ will be ``jump instructions'' for particles of the ASEP in the following way. If $(t; i, j) \in \mathcal{G}$ and there is a particle at site $i$ at time $t^-$, then the particle will attempt to jump to site $j$ at time $t$. If the destination is unoccupied at time $t^{-}$, the jump is performed; otherwise it is not. 

However, there remains the issue that infinitely many arrows are contained in any horizontal strip $\{ (x, y) \in \mathbb{Z} \times \mathbb{R}_{\ge 0} : y \in [a, b] \}$. Thus, we will show that we can partition $\mathbb{Z}$ into infinitely many finite intervals that do not ``interact.'' 

In what follows, we call a site $i$ \emph{inactive} if $\mathcal{L}^{(i)} \cap [0, T]$ and $\mathcal{R}^{(i)} \cap [0, T]$ are empty for each $j \in \{ i - 1, i, i + 1 \}$. If site $i$ is not inactive, then it is \emph{active}. 

\begin{lem}

\label{noarrows}

Fix some positive real number $T$. Almost surely, there exists a random integer sequence $\cdots < n_{-2} < n_{-1} < 0 < n_0 < n_1 < n_2 < \cdots$ such that each $n_j$ is inactive.   
\end{lem}

\begin{proof} 

For any integer $j$, the sets $\mathcal{L}^{(j)}$ and $\mathcal{R}^{(j)}$ are independent. Therefore, $\mathbb{P} \big[ \big| \big( \mathcal{R}^{(j)} \cup \mathcal{L}^{(j)} \big) \cap [0, T] \big| \ge 1\big] = 1 - e^{-T (L + R)}$. Since the sets $\{ \mathcal{L}^{(j)} \cup \mathcal{R}^{(j)} \}_{j \in \mathbb{Z}}$ are mutually independent, we deduce that the probability that any fixed site $i$ site is active is $c = 1 - e^{-3 T (L + R)} < 1$.  

Observe that, for any $i \in \mathbb{Z}$, the events that site $i + 3k$ is inactive are mutually independent (where $k$ ranges across $\mathbb{Z}$). Thus, the probability that all integers $i \in [4^k, 4^{k + 1}]$ are active is at most equal to $c^{4^k}$. Since $\sum_{k = 0}^{\infty} c^{4^k} < \infty$, the Borel-Cantelli lemma shows that there almost surely exist integers $0 < k_1 < k_2 < \cdots $ satisfying the following. For each positive integer $j$, there is some $n_j \in [4^{k_j}, 4^{k_j + 1}]$ that is inactive. This defines the subsequence $n_0, n_1, n_2, \ldots $. The complementary subsequence sequence $n_{-1}, n_{-2}, \ldots $ can be defined similarly. 
\end{proof}

Lemma \ref{noarrows} states that we can partition $\mathbb{Z}$ into the disjoint union of intervals $\bigsqcup_{j \in \mathbb{Z}} I_j$, where $I_j = [n_j, n_{j + 1} - 1]$, such that any particle that begins in some $I_j$ remains in $I_j$ throughout the time interval $[0, T]$. Thus, the ASEP on $\mathbb{Z}$ can be defined up to time $T$ through the following four steps. 

\begin{enumerate}

\item{ Form the ASEP time graph $\mathcal{G}$.} 

\item{ Using Lemma \ref{noarrows}, find inactive sites $\cdots , n_{-1} < n_0 < n_1 < \cdots $. Partition $\mathbb{Z} = \bigsqcup_{j \in \mathbb{Z}} I_j$.} 

\item{Use $\mathcal{G}$ to run the ASEP on each finite interval $I_j$ up to time $T$; this is well-defined since $I_j$ contains only finitely many particles.}

\item{Concatenate the processes on the intervals $I_j$ to produce the ASEP on $\mathbb{Z}$.}

\end{enumerate}

\subsection{The Time Graph of the Offset Stochastic Six-Vertex Model}

\label{OffsetGraph}

In this section we define a \emph{discrete time graph} and use it to resample the offset stochastic six-vertex model.

Associated with each integer $i \in \mathbb{Z}$ will be two random sets $\mathcal{L}^{(i)}, \mathcal{R}^{(i)} \subset \mathbb{Z}_{> 0} \times \mathbb{Z}$, such that $\mathcal{L}^{(i)}$ and $\mathcal{R}^{(i)}$ each contain at most one element of the form $(t, j)$ for any fixed $t \in \mathbb{Z}_{> 0}$; here, $t$ will refer to the time of an attempted jump and $j$ will refer to the destination of the attempted jump. All sets will be mutually independent. 

The sets are sampled as follows. For each fixed $t \in \mathbb{Z}_{> 0}$, we append $(t, i - 1)$ to $\mathcal{L}^{(i)}$ with probability $\delta_1$; otherwise, with probability $1 - \delta_1$, there is no element of the form $(t, j)$ in $\mathcal{L}^{(i)}$. Similarly, for each fixed $t \in \mathbb{Z}_{> 0}$ and $k \in \mathbb{Z}_{> 0}$, we append $(t, i + k)$ to $\mathcal{R}^{(i)}$ with probability $(1 - \delta_2) \delta_2^k$; otherwise, with probability $1 - \delta_2$, there is no element of the form $(t, j)$ in $\mathcal{R}^{(i)}$. 

Define the \emph{discrete time graph} $\mathcal{D} = \bigsqcup_{i \in \mathbb{Z}} \big( \bigsqcup_{(t, j) \in \mathcal{R}^{(i)}} \{ (t; i, j) \} \cup \bigsqcup_{(t, i - 1) \in \mathcal{L}^{(i)}} \{ (t; i, i - 1) \} \big) $. 

In what follows, we say that a particle at site $i$ \emph{tries to jump to site $j > i$ at time $t$} if it jumps to the minimal site $m \in [i, j]$ such that $m + 1$ was occupied at time $t - 1$; if no such $m$ exists, then the particle jumps to site $j$. 

Now let us use our time graph $\mathcal{D}$ to sample the offset stochastic six-vertex model. We omit the verification that this is indeed a resampling of the offset stochastic six-vertex model, as this follows directly from the description of the dynamics of the model given in Section \ref{Spin12}. 

The model is initialized as in the first step in Section \ref{Spin12}, by placing a (red) particle at site $i \in \mathbb{Z}_{\ge 1}$ at time $t = 0$ if and only if $\varphi_i^{(x)} = 1$, and then by ordering the positions of the particles $q_0 (0) < q_1 (0) < \cdots$ and setting $N(0) = 0$.

\begin{enumerate}

\item{ \label{1} At time $t \ge 1$, set $N(t) = N(t - 1) + 1$ if and only if $\varphi_t^{(y)} = 1$; that is, add a new (blue) particle to the system at time $t$ if and only if $\varphi_t^{(y)} = 1$. 

\begin{enumerate}
\item{ \label{1a} If $\varphi_t^{(y)} = 1$ and there does not exist any $(t, j) \in \mathcal{R}^{(1 - t)}$, then set $q_{-N(t)} (t) = 1 - t$. }

\item{ \label{1b} Assume that $\varphi_t^{(y)} = 1$ and that there exists some (and thus only one) $(t, j) \in \mathcal{R}^{(1 - t)}$. Then, particle $-N(t)$ tries to jump from site $1 - t$ to $j$ at time $t$. }

\end{enumerate}
}

\item{ \label{2} Now let $i > - N(t)$. Conditioned on $q (t - 1)$ and $q_{i - 1} (t)$, particle $i$ jumps as follows. 

\begin{enumerate}

\item{ \label{2a} Assume that $\big( t, q_i (t - 1) - 1 \big) \in \mathcal{L}^{q_i (t - 1)}$. 

\begin{enumerate}

\item{ \label{2a1} If site $q_i (t - 1) - 1$ is unoccupied at time $t$ (equivalently, if $q_{i - 1} (t) < q_i (t - 1) - 1$), then particle $i$ jumps left, so set $q_i (t) = q_i (t - 1) - 1$. }

\item{ \label{2a2} Otherwise (equivalently, if $q_{i - 1} (t) = q_i (t - 1) - 1$), particle $i$ does not jump left. In this case, we follow step \ref{2b1} and step \ref{2b2} below (the particle can still jump right if it tries but does not succeed to jump left).}

\end{enumerate}

}

\item{ \label{2b} Assume that $\big( t, q_i (t - 1) - 1 \big) \notin \mathcal{L}^{q_i (t - 1)}$.

\begin{enumerate}

\item{ \label{2b1} If no element of the form $(t, j)$ is in $\mathcal{R}^{q_i (t - 1)}$, then particle $i$ does not move and we set $q_i (t) = q_i (t - 1)$.}

\item{ \label{2b2} Assume that there exists some (and thus only one) $(t, j) \in \mathcal{R}^{q_i (t - 1)}$. Then, particle $i$ tries to jump from site $q_i (t - 1)$ to $j$ at time $t$.} 

\end{enumerate} }

\end{enumerate}  }

\end{enumerate}

\section{The \texorpdfstring{$[-M, N]$}{[-M, N]}-Bounded Stochastic Six-Vertex Model and ASEP}

\label{BoundedModelProcess}

As mentioned at the end of Section \ref{Spin12}, the presence of infinitely many particles in the ASEP and stochastic six-vertex model poses trouble for a direct proof of Theorem \ref{vertexmodelprocess}. To circumvent this issue, we introduce certain ``bounded versions'' of these two models. 

To define these models, we will fix positive integers $M$ and $N$, and recall the time graphs $\mathcal{G}$ and $\mathcal{D}$ from Section \ref{InfiniteExclusion} and Section \ref{OffsetGraph}. We will define certain truncated versions of $\mathcal{G}$ and $\mathcal{D}$, and then use these to define the bounded ASEP and stochastic six-vertex model. 

Specifically, form \emph{the $[-M, N]$-bounded time graph $\mathcal{G}^{[-M, N]}$} from $\mathcal{G}$ by removing all elements of the form $(t; i, j)$ with $i \notin [-M, N]$. Define \emph{the $[-M, N]$-bounded time graph $\mathcal{D}^{[-M, N]}$} from $\mathcal{D}$ in a similar way. 

Now, we define the \emph{$[-M, N]$-bounded ASEP with initial data $\varphi$} to be an interacting particle system $\big( \ldots , X_{-1}^{[-M, N]} (t), X_0^{[-M, N]} (t), X_1^{[-M, N]} (t), \ldots \big)$ on $\mathbb{Z}$ that follows the same dynamics as the ASEP with initial data $\varphi$, except that particles jump according to the $[-M, N]$-bounded time graph $\mathcal{G}^{[-M, N]}$ (instead of according to the original time graph $\mathcal{G}$). The $[-M, N]$-bounded ASEP coincides with the ASEP, modified so that no particle outside the interval $[-M, N]$ performs any jumps. 

Similarly, the \emph{$[-M, N]$-bounded offset stochastic six-vertex model with initial data $\varphi$} is an interacting particle system $\big( \ldots , q_{-1}^{[-M, N]} (t), q_0^{[-M, N]} (t), q_1^{[-M, N]} (t), \ldots \big)$ on $\mathbb{Z}$ that follows the same dynamics as the offset stochastic six-vertex model defined in Section \ref{OffsetGraph}, except that particles jump according to the $[-M, N]$-bounded discrete time graph $\mathcal{D}^{[-M, N]}$. 

We now will establish the convergence of the offset stochastic six-vertex model to the ASEP as follows. First, we show that the $[-M, N]$-bounded ASEP converges to the ASEP as $M$ and $N$ tend to $\infty$, and also establish a similar statement for the offset stochastic six-vertex model. Then, we show that the $[-M, N]$-bounded offset stochastic six-vertex model (with $\delta_1 = \varepsilon L$ and $\delta_2 = \varepsilon R$) converges to the $[-M, N]$-bounded ASEP, as $\varepsilon$ tends to $0$. The following three propositions provide more explicit statements. 

\begin{prop}

\label{boundedprocess} 

Consider the tagged particles $X(t) = \big( \ldots,  X_{-1} (t), X_0 (t), X_1 (t), \ldots \big)$ of the ASEP, as in Theorem \ref{vertexmodelprocess}. For each $M, N \in \mathbb{Z}_{> 0}$, denote the tagged particles of the corresponding $[-M, N]$-bounded ASEP by $X^{[-M, N]} (t) = \big( \ldots,  X_{-1}^{[-M, N]} (t), X_0^{[-M, N]} (t), X_1^{[-M, N]} (t), \ldots \big)$. 

Fix a positive integer $k$; integers $i_1, i_2, \ldots , i_k$; positive real numbers $t_1, t_2, \ldots , t_k$; and a finite set $S \subset \mathbb{Z}$. Then,
\begin{flalign*}
\displaystyle\lim_{M, N \rightarrow \infty} \mathbb{P} \big[ X_{i_1}^{[-M, N]} & (t_1), X_{i_2}^{[-M, N]} (t_2), \ldots , X_{i_k}^{[-M, N]} (t_k) \in S \big] = \mathbb{P} \big[  X_{i_1} (t_1), X_{i_2} (t_2), \ldots , X_{i_k} (t_k)  \in S \big]. 
\end{flalign*}

\end{prop}

\begin{prop}
	
\label{boundedvertex} 

Consider the offset stochastic six-vertex model $q(t) = \big( q_{-N(t)} (t), q_{1 - N(t)} (t), \ldots \big)$, as in Theorem \ref{vertexmodelprocess}. For each $M, N \in \mathbb{Z}_{> 0}$, let $q^{[-M, N]} (t) = \big( q_{-N(t)}^{[-M, N]} (t), q_{1 - N(t)}^{[-M, N]} (t), \ldots \big)$ be the corresponding $[-M, N]$-bounded offset stochastic six-vertex model. 

Fix a positive integer $k$; integers $i_1, i_2, \ldots , i_k$; positive real numbers $t_1, t_2, \ldots , t_k$; and a finite set $S \subset \mathbb{Z}$. Then,
\begin{flalign}
\label{vertexconvergencemn}
\begin{aligned} 
\displaystyle\lim_{M, N \rightarrow \infty} \mathbb{P} \Big[ q_{i_1}^{[-M, N]} & (\lfloor \varepsilon^{-1} t_1 \rfloor), q_{i_2}^{[-M, N]} (\lfloor \varepsilon^{-1} t_2 \rfloor), \ldots , q_{i_k}^{[-M, N]} (\lfloor \varepsilon^{-1} t_k \rfloor) \in S \Big] \\
& = \mathbb{P} \Big[ q_{i_1} (\lfloor \varepsilon^{-1} t_1 \rfloor), q_{i_2} (\lfloor \varepsilon^{-1} t_2 \rfloor), \ldots , q_{i_k} (\lfloor \varepsilon^{-1} t_k \rfloor) \in S \Big], 
\end{aligned} 
\end{flalign}

\noindent where the convergence is uniform in $\varepsilon > 0$. 

\end{prop}

\begin{prop}

\label{boundedvertexmodelprocess} 

Fix positive integers $M$ and $N$ and adopt the notation from Proposition \ref{boundedprocess} and Proposition \ref{boundedvertex}. Then, 
\begin{flalign*}
\displaystyle\lim_{\varepsilon \rightarrow 0} \mathbb{P} \Big[ q_{i_1}^{[-M, N]} & (\lfloor \varepsilon^{-1} t_1 \rfloor), q_{i_2}^{[-M, N]} (\lfloor \varepsilon^{-1} t_2 \rfloor), \ldots , q_{i_k}^{[-M, N]} (\lfloor \varepsilon^{-1} t_k \rfloor) \in S \Big] \nonumber \\ 
& = \mathbb{P} \Big[ X_{i_1}^{[-M, N]} (t_1), X_{i_2}^{[-M, N]} (t_2), \ldots , X_{i_k}^{[-M, N]} (t_k) \in S \Big]. 
\end{flalign*}

\end{prop}

Now Theorem \ref{vertexmodelprocess} follows directly from Proposition \ref{boundedprocess}, Proposition \ref{boundedvertex}, and Proposition \ref{boundedvertexmodelprocess} (using the uniformity of \eqref{vertexconvergencemn} in $\varepsilon$). Thus, it suffices to establish the three statements listed above. We address Proposition \ref{boundedprocess} and Proposition \ref{boundedvertex} in Section \ref{BoundedConverge} and Proposition \ref{boundedvertexmodelprocess} in Section \ref{ModelProcessConverge}.

\section{Proofs of Proposition \ref{boundedprocess} and Proposition \ref{boundedvertex}} 

\label{BoundedConverge}

In this section we establish Proposition \ref{boundedprocess} and Proposition \ref{boundedvertex}. To do this, we will show that there exists a large sub-interval of $[-M, N]$ that does not ``interact'' with its complement, for both the ASEP and the offset stochastic six-vertex model. We first address the ASEP. 

\subsection{Proof of Proposition \ref{boundedprocess}} 

\label{ProofProcess} 

Let us consider the ASEP as in Proposition \ref{boundedprocess}. We say that \emph{no jump is attempted through site $m \in \mathbb{Z}$ before time $T > 0$} if the time graph $\mathcal{G}$ of the ASEP contains no elements of the form $(t; i, j)$ with $t \in [0, T]$ and either $i \le m \le j$ or $j \le m \le i$. 

We have the following result. 

\begin{lem}

\label{interactionprocess}

Fix a positive real number $t$ and positive integers $M$ and $N$. Consider the ASEP as in Theorem \ref{vertexmodelprocess}. There exists a constant $c = c (t, L, R) > 0$ such that, with probability at least $1 - c^{M / 16} - c^{N / 16}$, there exist integers $-m \in [-M, -M / 16]$ and $n \in [N / 16, N]$ so that no jump is attempted through either $-m$ or $n$ before time $t$. 
\end{lem}

\begin{proof}

In what follows, we recall the notation and discussion in Section \ref{InfiniteExclusion}. In particular, we recall the notion of \emph{active} and \emph{inactive} sites from directly above Lemma \ref{noarrows}. Furthermore, we recall from the proof of Lemma \ref{noarrows} that there exists some constant $c < 1$ such that the probability that there is no inactive site in the interval $[4^k, 4^{k + 1}]$ is less than $c^{4^k}$, for each $k \in \mathbb{Z}_{> 0 }$. Thus, the probability that there is no inactive site in the interval $[N / 16, N]$ is less than $c^{N / 16}$; similarly, the probability that there is no inactive site in the interval $[-M, - M / 16]$ is less than $c^{M / 16}$. 

Hence, with probability at least $1 - c^{M / 16} - c^{N / 16}$, there exists an inactive site $-m \in [-M, - M / 16]$ and $n \in [N / 16, N]$. Since $-m$ and $n$ are inactive, no jumps can be attempted through $-m$ or $n$ before time $t$; this implies the lemma. 
\end{proof}

\noindent Now we can establish Proposition \ref{boundedprocess}. 

\begin{proof}[Proof of Proposition \ref{boundedprocess}]

Let us couple the time graph $\mathcal{G}$ of the ASEP and time graph $\mathcal{G}^{[-M, N]}$ of the $[-M, N]$-bounded ASEP so that $\mathcal{G}^{[-M, N]}$ is formed from $\mathcal{G}$ by removing all elements of the form $(t; i, j)$ with $i \notin [-M, N]$. 

Now, let $T > \max_{1 \le i \le k} t_i$ be a positive real number. Using Lemma \ref{interactionprocess} we produce, with probability at least $1 - c^{M / 32} - c^{N / 32}$, integers $-m \in [-M / 2, -M / 32]$ and $n \in [N / 32, N / 2]$ such that no jump is attempted through either $- m$ or $n$ before time $T$. 

This implies that, with probability at least $1 - c^{M / 32} - c^{N / 32}$, the restrictions to $[-M / 32, N / 32] \subseteq [-m, n]$ of the $[-M, N]$-bounded ASEP $X^{[-M, N]} (t)$ and the standard ASEP $X(t)$ coincide for all $t \in (0, T]$. Hence, the lemma follows by sending $M$ and $N$ to $\infty$, since $S \cup \{ X_{i_1} (0), X_{i_2} (0), \ldots , X_{i_k} (0) \} \subset [-M / 32, N / 32]$ with probability tending to $1$ as $M$ and $N$ tend to $\infty$. 
\end{proof}

\subsection{Proof of Proposition \ref{boundedvertex}}

\label{ProofVertex}

Let us define our notion of ``interaction'' for the stochastic six-vertex model. Similar to in Section \ref{ProofProcess}, we say that \emph{no jump is attempted through site $m \in \mathbb{Z}$ before time $T > 0$} if the time graph $\mathcal{D}$ of the stochastic six-vertex model contains no elements of the form $(t; i, j)$ with $t \in [1, T]$ and either $i \le m \le j$ or $j \le m \le i$. 

The following result is an analog of Lemma \ref{interactionprocess} for the stochastic six-vertex model. 

\begin{lem}

\label{interactionmodel} 

Fix a positive real number $t$ and positive integers $M$ and $N$. Let $\varepsilon > 0$ be a positive real number, and consider the stochastic six-vertex model from Theorem \ref{vertexmodelprocess}. 

There exists a positive constant $c = c(t, L, R) < 1$, independent of $\varepsilon$, such that the following holds. With probability at least $1 - c^{\min \{ M^{1 / 3}, N^{1 / 3} \} }$, there exist integers $- m \in [-M, -M / 2]$ and $n \in [N / 2, N]$ such no jump is attempted through either $- m$ or $n$ before time $\lfloor \varepsilon^{-1} t \rfloor$. 
\end{lem}

\begin{proof}
In what follows, we will assume that $M$ and $N$ are sufficiently large (for instance, greater than $10^9$), so that all inequalities below hold; this assumption amounts to enlarging the constant $c$ in the lemma. We denote $T = \lfloor \varepsilon^{-1} t \rfloor$. 

Observe that the statement of this lemma is similar to that of Lemma \ref{interactionprocess}. However, that lemma was quicker to establish since jumps in the ASEP are nearest-neighbor, a feature that considerably limits the number of ways in which a particle can jump through a given site. In contrast, jumps can be of arbitrarily large length in the stochastic six-vertex model, which permits particles far from the interval $[-M, -M / 2]$ to jump inside the interval $[-M, -M / 2]$; it also removes some of the independence that allowed for the proof of Proposition \ref{boundedprocess}. 

Thus, we will first bound the probability that particles to the left of $M$ attempt to jump into a smaller sub-interval of $[-M, -M / 2]$, and then we will bound the probability that particles in the interval $[-M, -M / 2]$ attempt to make very long jumps. Restricting to these events, we will be able to proceed as in the proof of Proposition \ref{boundedprocess} to conclude. 

To that end, fix $t', k \in \mathbb{Z}_{> 0}$ and $i \in \mathbb{Z}$, and observe that the probability that $\mathcal{D}$ contains a triplet of the form $(t'; i, j)$, with $j \ge i + k$, is bounded by $(1 - \delta_2) \sum_{j = k}^{\infty} \delta_2^j = \delta_2^k$. Therefore, the probability that $\mathcal{D}$ contains a triplet of the form $(t'; i, j)$, with $t' \in [1, T]$, $i \le -M$, and $j \ge -4 M / 5$, is bounded by $T \sum_{k = \lceil M / 5 \rceil}^{\infty} \delta_2^k \le T \delta_2^{M / 5} / (1 - \delta_2)$. Denote the event on which this occurs by $\mathcal{A}_1$. 

By similar reasoning, if we fix $t' \in \mathbb{Z}_{> 0}$ and $i \in \mathbb{Z}$, then the probability that $\mathcal{D}$ contains an element of the form $(t'; i, j)$ with $|i - j| > \sqrt{M}$, is bounded by $\delta_2^{\sqrt{M}}$. Thus the probability that $\mathcal{D}$ contains an element of the form $(t'; i, j)$, with $t' \in [0, T]$, $i \in [-M, 0]$, and $|i - j| \ge \sqrt{M}$, is most $(M + 1) T \delta_2^{\sqrt{M}} < 2 M T \delta_2^{\sqrt{M}} / (1 - \delta_2)$. Denote the event on which this occurs by $\mathcal{A}_2$. 

Denoting $\mathcal{A} = \mathcal{A}_1 \cup \mathcal{A}_2$, we find that $\mathbb{P} [\mathcal{A}] \le (1 - \delta_2)^{-1} \big( 2 M T \delta_2^{\sqrt{M}} + \delta_2^{M / 5} \big) < C' M \delta_2^{M^{1 / 3}}$, for some constant $C'$ independent of $\varepsilon$ (here, we used the facts that that $T < \varepsilon^{-1} t$ and $\delta_2 = \varepsilon R$ to deduce the independence from $\varepsilon$). Let us condition on the complement of the event $\mathcal{A}$. Thus $\mathcal{D}$ does not contain a triplet of the form $(t'; i, j)$ satisfying $t' \in [1, T]$ and $i \le -M < - 4 M / 5 \le j$, or satisfying $t' \in [1, T]$, $i \in [-M, 0]$, and $j \ge i + \sqrt{M}$. 

Now, fix an integer $i \in [-4M / 5, - M / 2)$. Fixing $t' \in \mathbb{Z}_{> 0}$, the probability that $\mathcal{D}$ contains a triplet of the form $(t'; i + 1, i)$ is equal to $\delta_1$. Similarly, the probability that $\mathcal{D}$ contains a triplet of the form $(t'; i_1, i_2)$ with $i_1 < i \le i_2$ is equal to $(1 - \delta_2) \sum_{k = 0}^{\infty} \sum_{j = 1}^{\infty} \delta_2^{j + k} < \delta_2 / (1 - \delta_2)$. By analogous reasoning, the probability that $\mathcal{D}$ contains a triplet of the form $(t'; i, i')$ (for some $i' \in \mathbb{Z}$) is bounded by $\delta_1 + \delta_2$. 

Therefore, the probability that a jump is attempted through a fixed site $i$ at some fixed time is at most equal to $2 \delta_1 + 2 \delta_2 / (1 - \delta_2)$. It follows that the probability that a jump is attempted through a fixed site $i$ sometime in the time interval $[1, T]$ is at most equal to $1 - \big( 1 - 2 \delta_1 - 2 \delta_2 / (1 - \delta_2)\big)^T < c'$, for some constant $c' < 1$ independent of $\varepsilon$, $M$, and $N$; to deduce the independence on $\varepsilon$, we used the facts that $\delta_1 = \varepsilon L$, $\delta_2 = \varepsilon R$, and $T = \lfloor \varepsilon^{-1} t \rfloor$. 

Now, recall that we are conditioning on the event $\mathcal{A}$, on which $\mathcal{D}$ does not contain any triplet of the form $(t'; i, j)$ with $t' \in [1, T]$ and $|i - j| > \sqrt{M}$. Thus, the events that a jump is attempted through site $i + k \lfloor 3 \sqrt{M} \rfloor$ are mutually independent as $k$ ranges across $\mathbb{Z}$. From this, we deduce that the probability that at least one jump is attempted through each site in the interval $[-4M / 5, - 3 M / 5]$ is at most equal to $c'^{M^{1 / 2} / 15} < c'^{M^{1 / 3}}$. 

Hence, the probability that at least one jump is attempted through each site in $[-M, -M / 2]$ before time $T$ is at most equal to $c'^{M^{1 / 3}}$, after conditioning on $\mathcal{A}$. Since $\mathbb{P} \big[ \mathcal{A} \big] \le C' M \delta_2^{M^{1 / 3}}$, we deduce that the probability $p_M$ that at least one jump is attempted through each site in $[-M, -M / 2]$ before time $T$ is at most $C' M \delta_2^{M^{1 / 3}} + c'^{M^{1 / 3}}$. Similarly, the probability $p_N$ that at least one jump is attempted through each site in $[N / 2, N]$ (before time $T$) is at most equal to $C'' N \delta_2^{N^{1 / 3}} + c''^{N^{1 / 3}}$, for some constants $C''$ and $c''$. 

Now, the lemma follows from summing $p_M$ and $p_N$ and taking the complement. 
\end{proof}

\noindent Now we can establish Proposition \ref{boundedvertex}. 

\begin{proof}[Proof of Proposition \ref{boundedvertex}]

As in the proof of Proposition \ref{boundedprocess}, we couple the time graph $\mathcal{D}$ of the offset stochastic six-vertex model and the time graph $\mathcal{D}^{[-M, N]}$ of the $[-M, N]$-bounded offset stochastic six-vertex model so that $\mathcal{D}^{[-M, N]}$ is formed from $\mathcal{D}$ by removing all elements of the form $(t; i, j)$ with $i \notin [-M, N]$. 

Now let $t > \max_{1 \le i \le k} t_i$ be a positive real number, and denote $T = \lfloor \varepsilon^{-1} t \rfloor$. By Lemma \ref{interactionprocess}, with probability tending to $1$ (independently of $\varepsilon$) as $M$ and $N$ tend to $\infty$, there exist integers $-m \in [-M, M / 2]$ and $n \in [N / 2, N]$ such that no jump is attempted through either $-m$ or $n$ before time $T$. Therefore, with probability tending to $1$ as $M$ and $N$ tend to $\infty$, the $[-M, N]$-bounded offset stochastic six-vertex model and the standard offset stochastic six-vertex model coincide up to time $T$ in the interval $[-M / 2, N / 2] \subseteq [-m, n]$. This implies the lemma, since the probability that $S \cup \{ q_{i_1} (0), q_{i_2} (0), \ldots , q_{i_k} (0) \} \subset [-M / 2, N / 2]$ tends to $1$ as $M$ and $N$ tend to $\infty$. 
\end{proof}

\section{Proof of Proposition \ref{boundedvertexmodelprocess}}

\label{ModelProcessConverge} 

In this section we establish Proposition \ref{boundedvertexmodelprocess} by first modifying the offset $[-M, N]$-bounded stochastic six-vertex model $q^{[-M, N]} (t)$ to form a process $\widetilde{q} (t) = \widetilde{q}^{[-M, N]} (t)$ that looks more similar to the $[-M, N]$-bounded ASEP. Then we will show that $q^{[-M, N]} (t)$ converges to $\widetilde{q}^{[-M, N]}$ and that $\widetilde{q}^{[-M, N]}$ converges to $X^{[-M, N]} (t)$, both as $\varepsilon$ tends to $0$.

\subsection{The Process \texorpdfstring{$\widetilde{q}$}{}}

\label{ProcessTilde}

Fix positive integers $M$ and $N$. Recall the definition of the discrete time graph $\mathcal{D}$ from Section \ref{OffsetGraph} and the $[-M, N]$-bounded discrete time graph $\mathcal{D}^{[-M, N]}$ from Section \ref{BoundedModelProcess}. We will define an interacting particle system $\widetilde{q} (t) = \widetilde{q}^{[-M, N]} (t)$ from $\mathcal{D}^{[-M, N]}$, as follows. We first form the \emph{altered time graph} $\widetilde{\mathcal{D}} = {\mathcal{D}}^{[-M, N]}$ from $\mathcal{D}^{[-M, N]}$ by replacing any triple $(t; i, j) \in \mathcal{D}$ with $(t; i, i + 1)$ if $j > i + 1$.

The model $\widetilde{q} (t)$ initialized slightly differently from the original offset stochastic six-vertex model $q (t)$. Namely, we place a particle at site $i \in \mathbb{Z}_{\ge 1}$ at time $t = 0$ if and only if $\varphi_i^{(x)} = 1$, and we place a particle at site $i \in \mathbb{Z}_{\le 0}$ at time $0$ if and only if $\varphi_{1 - i}^{(y)} = 1$; this matches the initialization for the ASEP given in Definition \ref{initialprocess}. Then, we order the positions of the particles $\cdots \widetilde{q}_{-1} (0) < \widetilde{q}_0 (0) < \widetilde{q}_1 (0) < \cdots$.

Now, let us describe the dynamics for $\widetilde{q} (t)$. We will have that $\widetilde{q} (0) = \widetilde{q} (1) = \cdots = \widetilde{q} (M)$, that is, the particle system $\widetilde{q}$ will not move for the first $M$ time steps. 

Suppose that $t > M$. The following describes how to update $\widetilde{q} (t - 1)$ to $\widetilde{q} (t)$, given $\widetilde{\mathcal{D}}$. 

\begin{enumerate}

\item{ \label{2tilde} Suppose that $(t; i, j) \in \widetilde{\mathcal{D}}$ for some $i, j \in \mathbb{Z}$. 

\begin{enumerate}

\item{ \label{2atilde} Suppose that $(t; i', j') \notin \widetilde{\mathcal{D}}$ for any $(i', j') \ne (i, j)$ such that $\{ i, j \} \cap \{ i', j' \}$ is non-empty. Then, the particle at site $i$ (at time $t - 1$) attempts to jump to site $j$. If the destination (site $j$) was unoccupied at time $t - 1$, then the jump is performed; otherwise, it is not and the particle does not move. }

\item{ \label{2btilde} Otherwise, if there exists such a $(t; i', j') \in \widetilde{\mathcal{D}}$, the particle at site $i$ (at time $t - 1$) does not move. } 

\end{enumerate} }

\item{ \label{3tilde} If there does not exist any $(t; i, j) \in \widetilde{\mathcal{D}}$, then set $\widetilde{q} (t) = \widetilde{q} (t - 1)$.}
\end{enumerate} 

Now, adopting the notation of Proposition \ref{boundedvertexmodelprocess}, consider the offset stochastic six-vertex model $\widetilde{q} (t)$ with parameters $\delta_1 = \varepsilon L$ and $\delta_2 = \varepsilon R$. Let us show that how $\widetilde{q} (t)$ converges to the tagged particle process for the bounded ASEP $X^{[-M, N]} (t)$. 

To that end, let $t > \max_{1 \le i \le k} t_i$ be a positive real number and denote $T = \lfloor \varepsilon^{-1} t \rfloor$. Define the \emph{rescaled time graph }$\mathcal{R}_{\varepsilon} \subset \mathbb{R} \times \mathbb{Z} \times \mathbb{Z}$ to be the set consisting of all triplets of the form $(\varepsilon t'; i, j)$, with $(t'; i, j) \in \widetilde{\mathcal{D}}$. 

Now, observe that the restriction $\mathcal{R}_{\varepsilon} \cap S_{M, N}$ of the rescaled time graph $\mathcal{R}_{\varepsilon}$ to the compact subset $S_{M, N} = [0, t] \times [-M, N] \times [-1 - M, N + 1] \subset \mathbb{R} \times \mathbb{Z} \times \mathbb{Z}$ converges to the restriction $\mathcal{G}^{[-M, N]} \cap S_{M, N}$ of the bounded time graph $\mathcal{G}^{[-M, N]}$ of the ASEP to $S_{M, N}$, as point processes on $S_{M, N}$. Indeed, this follows from the convergence of the Bernoulli trials process, with probability $p = \varepsilon s$ and number of trials $n = \lfloor \varepsilon^{-1} t \rfloor$, to the Poisson process with parameter $s$, run for time $t'$. 

Thus, since $\widetilde{q} (0) = X^{[-M, N]} (0)$ and since $\widetilde{q}$ and $X^{[-M, N]}$ evolve in the same way after their time graphs are given, it follows that
\begin{flalign}
\label{convergetildeq}
\begin{aligned}
\displaystyle\lim_{\varepsilon \rightarrow 0} \mathbb{P} \Big[ \widetilde{q_{i_1}} & (\lfloor \varepsilon^{-1} t_1 \rfloor), \widetilde{q_{i_2}} (\lfloor \varepsilon^{-1} t_2 \rfloor), \ldots , \widetilde{q_{i_k}} (\lfloor \varepsilon^{-1} t_k \rfloor) \in S \Big]  \\ 
& = \mathbb{P} \Big[ X_{i_1}^{[-M, N]} (t_1), X_{i_2}^{[-M, N]} (t_2), \ldots , X_{i_k}^{[-M, N]} (t_k) \in S \Big]. 
\end{aligned}
\end{flalign}

\subsection{Coupling \texorpdfstring{$q^{[-M, N]}$}{} and \texorpdfstring{$\widetilde{q}$}{}}	

\label{qtildeq}

Given \eqref{convergetildeq}, it suffices to identify the processes $\widetilde{q}$ and $q$ as $\varepsilon$ tends to $0$. To that end, we will show that if the processes $\widetilde{q}$ and $q^{[-M, N]}$ are coupled under the same time graph, then they look the same with high probability.

To explain further, we adopt the notation of Proposition \ref{boundedvertexmodelprocess}. Now, it is quickly verified that if there exists an integer $r \in [1, k]$ such that $\widetilde{q}_{i_r} (\lfloor \varepsilon^{-1} t_r \rfloor ) \ne q_{i_r}^{[-M, N]} (\lfloor \varepsilon^{-1} t_r \rfloor )$, then at least one of four possible events must occur. 

\begin{itemize}

\item{We have that $q_{i_r} (0) \le - \lfloor \varepsilon^{-1} t_r \rfloor $. }

\item{There exists some $t \in [1, M]$ and some $i \in [-M, N]$ such that $(t; i, j) \in \mathcal{D}$. }

\item{We have that $\widetilde{\mathcal{D}} \ne \mathcal{D}^{[-M, N]}$ when restricted to $[0, T] \times \mathbb{Z} \times \mathbb{Z}$. Equivalently, there is some $t \in [1, M]$, some $i \in [-M, N]$, and some $j \in \mathbb{Z}$ with $j \ge i + 2$, such that $(t; i, j) \in \mathcal{D}$. }

\item{There exist a time $t \in [1, T]$ and two distinct pairs $(i, j), (i', j') \in [-M, N] \times \mathbb{Z}$ such that $(t; i, j), (t; i', j') \in \mathcal{D}$.}

\end{itemize}

Since the $q_{i_r} (0)$ are random variables, the first event happens with probability tending to $0$ as $\varepsilon$ tends to $0$. 

To estimate probability of the second event, suppose that $t > 0$ and $i \in [-M, N]$ are fixed integers. Then, the probability that there exists a $j \in \mathbb{Z}$ such that $(t; i, j) \in \mathcal{D}$ is $\delta_1 + \delta_2$. Ranging $t \in [1, M]$ and $i \in [-M, N]$, we deduce that the probability of the second event is $M (M + N + 1) (\delta_1 + \delta_2) = M (M + N + 1) (R + L) \varepsilon$. 

To estimate the probability of the third event, again suppose that $t > 0$ and $i \in [-M, N]$ are fixed integers. Then, the probability that there exists a $j \in \mathbb{Z}$ such that $j \ge i + 2$ and $(t; i, j) \in \mathcal{D}$ is $\delta_2^2$. Ranging $t \in [1, \lfloor \varepsilon^{-1} \max_{1 \le r \le k} t_r \rfloor ]$ and $i \in [-M, N]$, we deduce that the probability of the second event is $(M + N + 1) \delta_2^2 \lfloor \varepsilon^{-1} \max_{1 \le s \le k} t_r \rfloor \le (M + N + 1) R^2 \varepsilon \max_{1 \le s \le k} t_s$.  

The probability of the fourth event can be estimated similarly to the probability of the third; it is bounded by $(M + N + 1)^2 (R + L)^2 \varepsilon \max_{1 \le s \le k} t_s$. 

Hence, the probability that at least one of the four events listed above occurs tends to $0$ as $\varepsilon$ tends to $0$. Thus,  
\begin{flalign}
\label{qconvergetildeq}
\begin{aligned}
\displaystyle\lim_{\varepsilon \rightarrow 0} \bigg( \mathbb{P} \Big[ \widetilde{q_{i_1}} & (\lfloor \varepsilon^{-1} t_1 \rfloor), \widetilde{q_{i_2}} (\lfloor \varepsilon^{-1} t_2 \rfloor), \ldots , \widetilde{q_{i_k}} (\lfloor \varepsilon^{-1} t_k \rfloor) \in S \Big]  \\ 
& - \displaystyle\lim_{\varepsilon \rightarrow 0} \mathbb{P} \Big[ q_{i_1}^{[-M, N]} (\lfloor \varepsilon^{-1} t_1 \rfloor), q_{i_2}^{[-M, N]}  (\lfloor \varepsilon^{-1} t_2 \rfloor), \ldots , q_{i_k}^{[-M, N]}  (\lfloor \varepsilon^{-1} t_k \rfloor) \in S \Big] \bigg) = 0. 
\end{aligned}
\end{flalign}

\noindent Now Proposition \ref{boundedvertexmodelprocess} follows from \eqref{convergetildeq} and \eqref{qconvergetildeq}.

\end{document}